\documentclass[A4]{article}
\usepackage[ansinew]{inputenc}
\usepackage{amsmath,amssymb,amsthm}

\newtheorem{theorem}{Theorem}[section]
\newtheorem{lemma}[theorem]{Lemma}

\newtheorem{corollary}[theorem]{Corollary}
\newtheorem{example}[theorem]{Example}
\newtheorem{proposition}[theorem]{Proposition}

\newtheorem{remark}[theorem]{Remark}
\newtheorem{remarks}[theorem]{Remarks}
\newtheorem*{notation}{Notation}

\newcommand{\F}{{\cal F}}

\newcommand{\defword}[1]{\emph{#1}}
\newcommand{\bino}[2]{\ensuremath{\genfrac{(}{)}{0pt}{1}{#1}{#2}}}

\DeclareMathOperator{\proj}{proj}

\author{Klaus Metsch\thanks{Justus-Liebig-Universit\"{a}t, Mathematisches Institut,
Arndtstra{\ss}e 2, D-35392 Gie{\ss}en}}
\begin{document}
\parindent0ex

\title{Erd\H os-Ko-Rado sets of flags of finite sets}

\parskip1ex

\maketitle

\begin{abstract}
A flag of a finite set $S$ is a set $f$ of non-empty proper subsets of $S$ such that $A\subseteq B$ or $B\subseteq A$ for all $A,B\in f$. The set $\{|A|:A\in f\}$ is called the type of $f$. Two flags $f$ and $f'$ are in general position (with respect to $S$) when $A\cap B=\emptyset$ or $A\cup B=S$ for all $A\in f$ and $B\in f'$. We study sets of flags of a fixed type $T$ that are mutually not in general position and are interested in the largest cardinality of these sets. This is a generalization of the classical Erd\H os-Ko-Rado problem. We will give some basic facts and determine the largest cardinality in several non-trivial cases. For this we will define graphs whose vertices are flags and the problem is to determine the independence number of these graphs.
\end{abstract}



\section{Introduction}

For integers $n\ge 0$ we define $[n]:=\{i\in\mathbb{Z}\mid 1\le i\le n\}$. Let $S$ be a non-empty finite set. We call two subsets of $S$ in \emph{general position} (with respect to  $S$) if their union is $S$ or their meet is empty, which means that their intersection is as small as possible. A \emph{flag} of $S$ is set $f$ of non-empty proper subsets of $S$ such that $X\subseteq Y$ or $Y\subseteq X$ for all $X,Y\in f$. The set $\{|X|:X\in f\}$ is the \emph{type} of $f$. Two flags $f$ and $f'$ of $S$ are in \emph{general position}, if $X$ and $X'$ are in general position for all $X\in f$ and all $X'\in f'$.
For $T\subseteq [|S|-1]$ the graph whose vertices are the flags of $S$ of type $T$ with two vertices being adjacent when the corresponding flags are in general position will be denoted by $\Gamma(S,T)$. If $S=[n]$, we use the notation $\Gamma(n,T)$. If $|T|=1$, these graphs are called \defword{Kneser graphs}.

The $q$-analogs of these graphs are obtained from vector spaces $F_q^n$ of dimension $n\ge 2$ where $F_q$ is the Galois field of order $q$. A \emph{flag} of $F_q^n$ is a set $f$ of subspaces such that $X\subseteq Y$ or $Y\subseteq X$ for all $X,Y\in f$, and the set $\{\dim(X)\mid X\in f\}$ is the \defword{type} of $f$. Two flags $f$ and $f'$ are in \emph{general position}, if $X+X'=F_q^n$ or $X\cap X'=\{0\}$ for all $X\in f$ and $X'\in f'$. For $T\subseteq[n-1]$, the graph whose vertices are the flags of type $T$ of $F_q^n$ with two vertices being adjacent when the corresponding flags are in general position will be denoted by $\Gamma_q(n,T)$.

In this paper we are interested in the independence number of the graphs $\Gamma(n,T)$ for $\emptyset\not=T\subseteq[n-1]$. For $|T|=1$ the problem reduces to the famous Erd\H os-Ko-Rado problem, which was solved in \cite{ErdHos1961}, but up to my knowledge the problem was not considered for $|T|>1$. In the $q$-analogue situation, the independence number of the graphs $\Gamma_q(n,T)$ with $|T|=1$ was determined by Frankl and Wilson \cite{Frankl&Wilson}, see also Hsieh \cite{Hsieh1975} and Newman \cite{Newman}. However, for the graphs $\Gamma_q(n,T)$ also several results for $|T|>1$ are known, see \cite{Sam}.

For every non-empty subset $T$ of $[n-1]$ the graph $\Gamma(n,T)$ is regular and non-empty, so its independence number is at most one half of the number of its vertices with equality if and only if the graph is bipartite. This is always the case when $T$ contains an element $i$ with $n-i\in T$. If $\max{T}\le \frac n2$ or $\min(T)\ge\frac n2$, then the independence number can easily be derived form the above mentioned result of Erd\H os-Ko-Rado. Section 2 is devoted to these facts as well as to several other basic facts.

There are pairs $(n,T)$ with $\emptyset\not=T\subseteq [n-1]$ for which it is to be a non-trivial problem to determine the independence number of $\Gamma(n,T)$. There are different techniques one might try, for example techniques that have been successful in the case $|T|=1$. One of the techniques is to apply the Hoffman bound, but for the cases I have tested, this bound is too weak when $|T|>1$ (except in some easy cases that are handled in Section 2 by different methods). One may consider the underlying coherent configuration and try to obtain a bound by semidefinite programming. But for the non-tivial cases I have tried this also does not give tight bounds. Another well-known technique is the cycle method and in fact a variant of this method will be used to show the first result of the paper. In order to state it, we denote the set of vertices of a graph $\Gamma$ by $V\Gamma$.

\begin{theorem}\label{main1}
For positive integers $b$ and $n$ and every non-empty subset $T$ of $[n]$ with $\max(T)+3b\le 2n<4b$, the independence number of the graph $\Gamma(n,T\cup\{b\})$ is equal to $\bino{n-1}{b} \cdot |V\Gamma(b,T)|$.
\end{theorem}

An independent set of this size of $\Gamma(n,T\cup\{b\})$ is given by all flags of type $T\cup\{b\}$ whose sets of cardinality $b$ are contained in $[n-1]$. We remark that for certain $T$, $b$ and $n$ there are non-equivalent independent sets of the same size. In fact, in Section 3 we construct several families of independent sets of $\Gamma(n,T)$ when $|T|=2$ and determine which of these are the largest ones. We will see that there are either one, two or three non-equivalent independents sets of maximal size among these examples. In Section 4 we show that these examples are in fact largest independent sets when $T=\{a,b\}$ with $b>\frac n2$ and $a+3b \le 2n$. The above theorem is a corollary of this result. In Section 5 we show in one other case, namely $T=\{1,n-2\}$ that the examples are best possible, which results in the following theorem.

\begin{theorem}
For $n\ge 5$ the independence number of $\Gamma(n,\{1,n-2\})$ is $\bino{n}{3}+2$ and there exist exactly three non-equivalent independent sets of this cardinality.
\end{theorem}

The proof is by induction on $n$ but we prove the induction step $n-1\to n$ only for $n\ge 10$. The induction basis $5\le n\le 9$ is verified by computer. In fact our result is more general, in that it proves an induction step in a more general setting. However, the proof of the induction basis for the other cases seems to be out of range for computers. Details are given in Section \ref{plus2}.

The above theorems combined with the results presented in Section 2 determine the independence number of all graphs $\Gamma(n,T)$ with $n\le 6$. Several open problems will be given in Section 6.

\section{Basic observations}

\begin{notation}\rm
\begin{enumerate}
\item For every positive integer $n$ we denote by $\omega_n$ the permutation of $[n-1]$ with $\omega_n(i)=n-i$ for all $i\in[n-1]$. For subsets $T$ of $[n-1]$, we define $\omega_n(T)=\{\omega_n(i)\mid i\in T\}$.
\item For every non-empty finite set $M$ and all sets $S\subseteq T\subseteq [|M|-1]$ we denote by $\proj^M_{T,S}$ the map from the set of all flags of $M$ of type $T$ to the set of all flags of $M$ of type $S$ that maps each flag $f$ of type $T$ to the flag $\{A\in f:|A|\in S\}$. If $M=[n]$, then we write $\proj^n_{T,S}$.
\item For each finite graph $\Gamma$ we denote its independence number by $\alpha(\Gamma)$.
\end{enumerate}
\end{notation}

\begin{lemma} \label{easycases1}
Let $n\ge 2$ be an integer and $\emptyset\not=T\subseteq [n-1]$. Then we have.
\begin{enumerate}
\renewcommand{\labelenumi}{\rm(\alph{enumi})}
\item The graph $\Gamma(n,T)$ is regular.
\item $\Gamma(n,T)\cong\Gamma(n,\omega_n(T))$ \emph{(duality)}.
\item $\alpha(\Gamma(n,T))\le \frac12|V\Gamma(n,T)|$  with equality if and only if $\Gamma(n,T)$ is bipartite.
\item If there exist $i,j\in T$ with $i+j=n$, then $\Gamma(n,T)$ is bipartite and $\alpha(\Gamma(n,T))=\frac12|V\Gamma(n,T)|$.
\end{enumerate}
\end{lemma}
\begin{proof}
(a) Since the symmetric group on the set $[n]$ induces a vertex transitive automorphism group of $\Gamma(n,T)$, this graph is regular.

(b) The map $f\mapsto \{[n]\setminus A\mid A\in f\}$ for every flag $f$ of $[n]$ induces a bijective map $\iota$ from the vertex set of $\Gamma(n,T)$ to the vertex set of $\Gamma(n,\omega_n(T))$. By the laws of DeMorgan two subsets of $[n]$ are in general position if and only if their complements in $[n]$ are in general position. Thus two flags of type $T$ are in general position if and only if their images under $\iota$ are in general position. Therefore $\iota$ is an isomorphism between the graphs.

(c) This is true for every finite regular non-empty graph.

(d) We remark that we allow that $i=j$. Every flag $f$ of type $\{i,j\}$ is in general position to exactly one flag of type $\{i,j\}$, namely $\{[n]\setminus A\mid A\in f\}$. If we take from every such pair one flag, we obtain a set $X$ consisting of flags of type $\{i,j\}$ that are mutually not in general position. The set consisting of all flags $f$ of type $T$ of $[n]$ that satisfy $\proj^n_{T,\{i,j\}}(f)\in X$ has the property that no two of its flags are in general position and moreover this set contains $\frac12|V\Gamma(n,T)|$ elements. By (c), the graph is bipartite.
\end{proof}

\begin{remark}
If there exists $i\in T$ with $j=n-i\in T$, then $\Gamma(n,T)$ is disconnected and each component is a complete regular bipartite graph.
\end{remark}

\begin{lemma}\label{projectinggeneral}
Let $S$ and $T$ be distinct non-empty subsets of $[n-1]$ with $S\subseteq T$.
\begin{enumerate}
\renewcommand{\labelenumi}{\rm(\alph{enumi})}
\item  Let $f_1$ and $f_2$ be flags of $[n]$ of type $T$. If $f_1$ and $f_2$ are in general position, then so are $\proj^n_{T,S}(f_1)$ and $\proj^n_{T,S}(f_2)$. In other words, $\proj^n_{T,S}$ is a graph homomorphism from $\Gamma(n,T)$ onto $\Gamma(n,S)$.
\item We have
    \begin{align}\label{eqn_projection1}
    \alpha(\Gamma(n,T))\ge \alpha(\Gamma(n,S))\cdot \frac{|V\Gamma(n,T))|}{|V\Gamma(n,S)|}.
    \end{align}
\end{enumerate}
\end{lemma}
\begin{proof}
(a) This follows immediately form the definition of general position.

(b) Let $\F$ be an independent set of $\Gamma(n,S)$ with $|\F|=\alpha(\Gamma(n,S))$. By (a), the preimage $\bar\F$ of $\F$ under the map $\proj^n_{T,S}$ is an independent set of $\Gamma(n,T)$. The fraction $v$ in \eqref{eqn_projection1} is the common number of preimages of every vertex of $\Gamma(n,S)$ under the projection map $\proj^n_{T,S}$. Hence $|\bar\F|=v\cdot \alpha(\Gamma(n,S))$.
\end{proof}

\begin{remark}
In general we do not have equality in \eqref{eqn_projection1}. For example for $n=6$, $S=\{4\}$ and $T=\{1,4\}$ we have $\alpha(\Gamma(n,S))=\bino{5}{4}=5$ and the fraction in \eqref{eqn_projection1} is $4$ whereas we shall see in Theorem \ref{nhochplus2} that $\alpha(\Gamma(n,T))=22$.
\end{remark}

Now we present a situation where \eqref{eqn_projection1} holds with equality.

\begin{proposition}\label{projectingspecial}
Let $S$ and $T$ be distinct non-empty subsets of $[n-1]$ satisfying
\begin{align*}
S\subseteq T,\ \min(S)+\max(S)\le n\ \text{and}\ \max(T\setminus S)<\min(S).
\end{align*}
Then two flags of $\Gamma(n,T)$ are in general position if and only if their images under $\proj^n_{T,S}$ are in general position. In particular
    \begin{align}\label{eqn_projection}
    \alpha(\Gamma(n,T))=\alpha(\Gamma(n,S))\cdot |V\Gamma(\min S,T\setminus S)|.
    \end{align}
\end{proposition}
\begin{proof}
Let $f_1$ and $f_2$ be flags of $\Gamma(n,T)$ and $g_1$ and $g_2$ their images under $\proj^n_{T,S}$. If $f_1$ and $f_2$ are in general position, then so are their images by the previous lemma. Now assume that $g_1$ and $g_2$ are in general position. Let $A_i,B_i\in g_i$ with $|A_i|=\min(S)$ and $|B_i|=\max(S)=\max(T)$. Since $\min(S)+\max(S)\le n$ and since $g_1$ and $g_2$ are in general position we have $A_1\cap B_2=\emptyset$. If $C_1\in f_1\setminus g_1$, then $C_1\subseteq A_1$, so $C_1\cap B_2=\emptyset$ and hence $C_1\cap C=\emptyset$ for all $C\in f_2$. Similarly $C_2\cap C=\emptyset$ for all $C_2\in f_2\setminus g_2$ and all $C\in f_1$. Therefore $f_1$ and $f_2$ are in general position.

Since $\max(T\setminus S)<\min(S)$, every vertex of $\Gamma(n,S)$ is the image of exactly $v:=|V\Gamma(\min(S),T\setminus S)|$ vertices of $\Gamma(n,T)$ under the map $\proj^n_{T,S}$. Therefore the $\ge$-part in \eqref{eqn_projection} follows from the previous lemma. On the other hand, if $\F$ is an independent set of $\Gamma(n,T)$, then the first statement of the present lemma shows that  $\proj^n_{T,S}(\F)$ is an independent set of $\Gamma(n,S)$. Since $|\F|\le v\cdot|\proj^n_{T,S}(\F)|$, this proves the $\le$-part of \eqref{eqn_projection1}.
\end{proof}

\begin{remarks}
\begin{enumerate}
\item The proof also shows that every largest independent set of $\Gamma(n,T)$ is the preimage of a largest independent set of $\Gamma(n,S)$.
\item Using Lemma \ref{easycases1} (b), the result of Proposition \ref{projectingspecial} can be dualized. Similarly, the next result can be dualized to the situation when $\min(T)\ge\frac n2$.
\end{enumerate}
\end{remarks}

\begin{corollary}\label{easycases2}
Consider $1<n\in\mathbb{Z}$ and $\emptyset\not=T\subseteq [n-1]$ with $t:=\max(T)\le\frac12n$. Then $\alpha(\Gamma(n,T))=\bino{n-1}{t-1}\cdot|V\Gamma(t,T\setminus\{t\})|$.
\end{corollary}
\begin{proof}
This results is the classical result of Erd\H os-Ko-Rado \cite{ErdHos1961} when $T=\{t\}$; in this case $\Gamma(t,T\setminus\{t\})$ is the empty graph with one vertex. Otherwise the corollary follows from the proposition using $S=T\setminus\{t\}$ and again the Erd\H os-Ko-Rado result.
\end{proof}

\begin{remark}
The results of  \ref{easycases1}, \ref{projectingspecial} and \ref{easycases2} cover quite some of the possible types. We list the open cases for $2\le n\le 8$ up to duality. For $n\le 4$, there is no open case. For $n=5$ this is only $\{1,3\}$. For $n=6$, this is $\{1,4\}$. For $n=7$, these are $\{1,4\}$, $\{1,5\}$, $\{2,4\}$, $\{1,2,4\}$, $\{1,3,5\}$ and $\{1,4,5\}$. For $n=8$, these are $\{1,5\}$, $\{1,6\}$, $\{2,5\}$, $\{1,2,5\}$, $\{1,3,6\}$ and $\{1,5,6\}$. Some of these cases will be handled in Sections \ref{cyclic} and \ref{plus2}.
\end{remark}

\begin{remark}
The results in this section determine the independence number of $\Gamma(n,T)$ in all cases in which the independence number for the graphs $\Gamma_q(n,T)$ is known.
\end{remark}
\section{Examples}

In this section we study examples of independent sets of the graphs $\Gamma(n,\{a,b\})$ with $a<b$. Since the case $\max\{a,b\}\le \frac n2$ and the dual case $\min\{a,b\}\ge \frac n2$ have been treated in Corollary \ref{easycases2}, we assume that $a<\frac n2<b$. The case $a+b=n$ has been treated in Lemma \ref{easycases1} (d), so that we can assume $a+b\not=n$. In view of Lemma \ref{easycases1} (b) we will restrict ourselves therefore to the situation when $a+b<n$. For simplicity a flag $\{A,B\}$ of type $\{a,b\}$ will mostly be written as an ordered pair $(A,B)$ with $|A|=a$ and $|B|=b$. Recall that $[0]=\emptyset$.

\begin{example}\label{examp}
Let $n$, $a$ and $b$ be positive integers with $a<\frac n2<b$ and $a+b<n$. For every integer $i$ with $0\le i\le 2b-n+1$ we denote by $\F_i(n,a,b)$ the set of all vertices $(A,B)$ of $\Gamma(n,\{a,b\})$ that fulfill at least one of the following conditions.
\begin{itemize}
\item[\rm (I)] $[i]\subseteq B\subseteq[n-1]$,
\item[\rm (II)] $\min A\le i$ and $[\min A]\subseteq B$.
\end{itemize}
For $i=2b-n+1$, let $\bar\F_i(n,a,b)$ be the set of all vertices $(A,B)$ of $\Gamma(n,\{a,b\})$ that satisfy $[i]\subseteq B$ or (II); hence $\F_i(n,a,b)\subseteq \bar\F_i(n,a,b)$.
\end{example}

\begin{lemma}\label{ismaximal}
Let $n,a,b$ be positive integers with $a+b<n$ and $b>\frac n2$.
\begin{enumerate}
\renewcommand{\labelenumi}{\rm(\alph{enumi})}
\item For $0\le i\le 2b-n$ the set $\F_i(n,a,b)$ is a maximal independent set of $\Gamma(n,\{a,b\})$.
\item For $i=2b-n+1$, the set $\bar\F_i(n,a,b)$ is a maximal independent set of $\Gamma(n,\{a,b\})$ and it is the only maximal independent that contains $\F_i(n,a,b)$. Moreover $|\bar\F_i\setminus \F_i|+\bino{2s-2}{s-2}\bino{s-1}{a}$ where $s=n-b$.
\end{enumerate}
\end{lemma}
\begin{proof}
We put $\F_i=\F_i(n,a,b)$. It is easy to see that all sets in question are independent sets. To investigate maximality we consider the situation that there exists an integer $i$ with $0\le i\le 2b-n+1$ and a flag $(A,B)$ such that $(A,B)\notin \F_i(n,a,b)$ and $\F_i\cup\{(A,B)\}$ is independent. We have to show that $i=2b-n+1$ and $(A,B)\in \bar\F_i(n,a,b)$.

First we want to show that $[i]\subseteq B$. Assume the contrary. Then there exists a positive integer $j\le i$ with $[j-1]\subseteq B$ and $j\notin B$. Since $(A,B)\notin\F_i$, then $\min(A)>j-1$ and hence $\min(A)\ge j+1$. We have $|B\setminus [j-1]|\ge |B|-(i-1)=b+1-i\ge n-b$. Hence there exists a subset $T$ of cardinality $n-b$ of $[n]$ with $A\subseteq T\subseteq B\setminus[j-1]$. Put $B':=[n]\setminus T$. Then $B\cup B'=[n]$ and $|B'\setminus B|=n-b\ge a$. Chose $A'\subseteq B'\setminus B$ with $|A'|=a$ and $j\in A'$. Then $(A,B)$ and $(A',B')$ are in general position. Since $\min(A')=j$ and $[j]\subseteq B'$ we have $(A',B')\in\F_i$. But $\F_i\cup\{(A,B)\}$ is independent, contradiction.

Therefore $[i]\subseteq B$. Since $(A,B)\notin \F_i$, then $\min(A)>i$ and $n\in B$. We want to show that $i=2b+1-n$. Assume on the contrary that $i\le 2b-n$. Then $|B\setminus [i]|=b-i\ge n-b$. As $a<n-b$, we thus find a subset $T$ of $B\setminus[i]$ with $A\cup\{n\}\subseteq T$ and $|T|=n-b$. Put $B':=[n]\setminus T$. Then $B\cup B'=[n]$ and $|B'\setminus B|=n-b\ge a$. Chose $A'\subseteq B'\setminus B$ with $|A'|=a$. Then $(A,B)$ and $(A',B')$ are in general position. Since $[i]\subseteq B'\subseteq [n-1]$ we have $(A',B')\in\F_i$. But $\F_i\cup\{(A,B)\}$ is independent, contradiction.

Hence $i=2b+1-n$. This proves (a). We have moreover shown that every flag $(A,B)$ that is not in $\F_i$ and such that $\F_i\cup\{(A,B)\}$ is independent satisfies $[i]\subseteq B$ and hence $(A,B)\in\bar\F_i(n,a,b)$. This is the first part of (b). Since the elements $(A,B)$ of $\bar\F_i(n,a,b)\setminus\F_i(n,a,b)$ are the flags with $[i]\cup\{n\}\subseteq B$ and $\min(A)\ge i+1$, we see that there exists $\bino{n-i-1}{b-i-1}\cdot\bino{b-i}{a}$ such flags.
\end{proof}

\begin{lemma}\label{SizeFi}
Let $n,a,b$ be positive integers with $a+b<n$ and $b>\frac n2$. For $0\le i\le 2b-n+1$, the cardinality of $\F_i(n,a,b)$ is
\begin{align*}
\binom{n-1-i}{b-i}\binom{b-i}{a}+\binom{n-b+a-1}{a-1}\left(\binom{n}{b-a}-\binom{n-i}{b-a-i}\right).
\end{align*}
Moreover, with
\begin{align}\label{eq_def_i0}
i_0:=b-1-\frac{(n-b)(n-b-1)}{a}
\end{align}
we have the following. If $i_0<0$, then the largest cardinality occurs only for $i=0$. If $i_0\ge 0$ is an integer, then the largest cardinality occurs for $i\in\{i_0,i_0+1\}$. If $i_0\ge 0$ but $i_0$ is not an integer, then the largest cardinality    occurs only for  $i=\lceil i_0\rceil$.
\end{lemma}
\begin{proof}
First notice that the condition $a+b<n$ implies that $i_0\le 2b-n-1$. The number of flags satisfying only the first condition of Example \ref{examp} is $\bino{n-1-i}{b-i}\bino{b-i}{a}$. With the notation $j:=\min(A)$ in the second condition of Example \ref{examp} we find
\begin{align*}
|\F_i(n,a,b)|&=\binom{n-1-i}{b-i}\binom{b-i}{a}+\sum_{j=1}^i\binom{n-j}{a-1}\binom{n-a-j+1}{b-a-j+1}
\\
&=\binom{n-1-i}{b-i}\binom{b-i}{a}+\sum_{j=1}^i\binom{n-b+a-1}{n-b}\binom{n-j}{n-b+a-1}
\\
&=\binom{n-1-i}{b-i}\binom{b-i}{a}+\binom{n-b+a-1}{a-1}\left(\binom{n}{b-a}-\binom{n-i}{b-a-i}\right)
\end{align*}
This proves the formula for $\F_i(n,a,b)$. We also find
\begin{align}\label{eq_good_formul_Fi}
|\F_i(n,a,b)|=\binom{n-b+a-1}{a-1}\binom{n}{b-a}+\frac{((n-b)^2+a(i-b))(n-1-i)!}{(n-b+a)a!(n-b)!(b-i-a)!}.
\end{align}
Using this, an easy calculation shows that $|\F_{i+1}|\ge|\F_i|$ if and only if
$i\le i_0$ and that $|\F_{i+1}|=|\F_i|$ if and only if $i=i_0$. This implies the second statement of the lemma.
\end{proof}

\textbf{Definition.}
We call two independent sets of a graph \emph{equivalent}, if the are in the same orbit under the automorphism group of the graph.

\begin{proposition}\label{nonequivalent}
With $a,b,n$ and $i_0$ as in Lemma \ref{SizeFi} we have the following. \renewcommand{\labelenumi}{\rm(\alph{enumi})}
\begin{enumerate}
\item
If $i_0$ is an integer, then $\F_{i_0}(n,a,b)$ and $\F_{i_0+1}(n,a,b)$ are non-equivalent independent sets of $\Gamma(n,\{a,b\})$.
\item If $a+b=n-1$, then $i_0=2b-n-1$ and the sets $\F_{i_0}(n,a,b)$, $\F_{i_0+1}(n,a,b)$ and $\bar\F_{i_0+2}(n,a,b)$ are non-equivalent independent set of $\Gamma(n,\{a,n-a-1\})$ of size $\bino{n}{2a+1}\bino{2a}{a-1}+\bino{2a}{a}$.
\end{enumerate}
\end{proposition}
\begin{proof}
Put $\F_i=\F_i(n,a,b)$ for $i\in\{i_0,i_0+1\}$ and $\F_i=\bar\F_i(n,a,b)$ for $i=i_0+2$.

(a) \underline{Part 1}. Consider $i\in\{i_0,i_0+1\}$ and
a vertex $(A,B)$ that does not lie in $\F_i$. In this part we determine the number of neighbors of $(A,B)$ in $\F_i$ or a lower bound for this number.

\underline{Case 1}. $[i]\subseteq B$.

As $(A,B)\notin\F_i(n,a,b)$, it follows that $\min(A)\ge i+1$ and $n\in B$. With
\begin{align*}
z_1(i):=\bino{n-i-1}{n-b}\bino{b-i}{a}\ \text{and}\ z_2(i):=\bino{n-i-1}{n-b}\bino{b-i-1}{a}
\end{align*}
there exist exactly $z_1(i)$ such vertices $(A,B)$ and exactly $z_2(i)$ have $n\notin A$.
The neighbors of $(A,B)$ in $\F_i$ are the vertices $(A',B')$ satisfying $[i]\subseteq B' \subseteq [n-1]$ with $|B\cap B'|=2b-n$ and $B'\cap A=\emptyset$ and $A'\subseteq B'\setminus B$. It follows the
that number of neighbors of $(A,B)$ in $\F_i$ is
\begin{align}
m_1(i):=\binom{b-i-a-1}{n-b-a-1}\binom{n-b}{a}, \ \ \text{if $n\notin A$}\label{num2},
\\
m_2(i)=\binom{b-i-a}{n-b-a}\binom{n-b}{a}, \ \ \text{if $n\in A$}\label{num3}.
\end{align}

\underline{Case 2}. $[i]\not\subseteq B$.

Then $[j-1]\subseteq B$ and $j\notin B$ for some integer $j\in[i]$. As $(A,B)\notin\F_i(n,a,b)$, then $A\subseteq B\setminus[j-1]$. Therefore $|B\setminus(A\cup[j-1])|=b-a+1-j$.
Therefore the number of neighbors $(A',B')$ of $(A,B)$ in $\F_i(n,a,b)$ that satisfy $[j]\subseteq B'$ and $j=\min A'$ is
\begin{align*}
m(j):=\binom{b-a+1-j}{n-b-a}\binom{n-b-1}{a-1}.
\end{align*}
If $n\in B$, then $(A,B)$ has neighbors $(A',B')$ in $\F_i$ that satisfy $[i]\subset B'\subseteq[n-1]$ and $\min(A')>i$. Hence the number of neighbors of $(A,B)$ in $\F_i(n,a,b)$ is at least $m_3(i):=m(i)$ and equality occurs if and only if $j=i$ and $n\notin B$. The number of such vertices $(A,B)$ is \begin{align*}
z_3(i):=\binom{n-i-1}{n-b-2}\binom{b+1-i}{a}.
\end{align*}

\underline{Part 2}. Now we use this information to show that $\F_{i_0}$ and $\F_{i_0+1}$ are not equivalent. Recall the definition of $i_0$ in \eqref{eq_def_i0} and that $a+b\le n-1$. This implies that $m_2(i_0)>m_1(i_0)$ and $m_3(i_0)>m_1(i_0)$. Hence, every vertex outside $\F_{i_0}$ has at least $m_1(i_0)$ neighbors in $\F_{i_0}$ and there are exactly $z_2(i_0)$ vertices with $m_1(i_0)$ neighbors in $\F_{i_0}$.

We have seen that there are vertices outside $\F_{i_0+1}$ with $m_1(i_0+1)$ neighbors in $\F_{i_0+1}$.
If $a+b<n-1$, then $m_1(i_0)>m_1(i_0+1)$ and hence $F_{i_0}$  and $F_{i_0+1}$ are not equivalent.

Now consider the case $a+b=n-1$. Then $i_0=2b-n-1$ and $m_1(i_0)=m_1(i_0+1)=m_2(i_0+1)$. Also, using $a+b=n-1$, we see that $m_3(i_0+1)\ge m_1(i_0+1)$ with equality if and only if $a=1$. If $a>1$, it follows that $z_1(i_0+1)$ vertices have exactly $m_1(i_0)$ neighbors in $F_{i_0+1}$. Since $z_1(i_0+1)<z_2(i_0)$, it follows in this case that $F_{i_0}$  and $F_{i_0+1}$ are not equivalent. Finally consider the case $a=1$ and $b=n-2$. Then $i_0=n-5$ and $m_1(i_0)=m_k(i_0+1)$ for $k=1,2,3$. Also, there are exactly $z_1(i_0+1)+z_3(i_0+1)=6+3=9$ vertices with $m_1(i_0)$ neighbors in $F_{i_0+1}$ and $z_2(i_0)=12$ vertices with $m_1(i_0)$ neighbors in $F_{i_0}$. Thus, also in this situation, the two independent sets are not equivalent.

(b) Put $\F_{i_0+2}=\bar\F_{i_0+2}(n,a,b)$. In this part we have $a+b=n-1$, which implies that $i_0=2b-n-1=b-a-2\ge 0$, $m_1(i_0+1)=m_2(i_0+1)=a+1$ and $m_3(i_0+1)=2a$. From the proof of (a) we know the following:
\begin{itemize}
\item The vertices of $M_1:=\{(A,B)\mid[i_0+1]\cup\{n\}\subseteq B,\ \min(A)\ge i_0+2\}$ have $a+1$ neighbors in $\F_{i_0+1}$ and $|M_1|=z_1(i_0+1)$.
\item The vertices of $M_2:=\{(A,B)\mid[i_0]\subseteq B,\ i_0+1,n\notin B,\ \min(A)>i_0+1\}$ have  $2a$ neighbors in $\F_{i_0+1}$ and $|M_2|=z_3(i_0+1)$.
\item Every other vertex outside $\F_{i_0+1}$ has more than $2a$ neighbors in $\F_{i_0+1}$.
\end{itemize}
In a similar way one sees the following.

\begin{itemize}
\item The vertices of $M'_1:=\{(A,B)\mid [i_0+1]\subseteq B,\ i_0+2\notin B,\ \min(A)>i_0+2\}$ have $a+1$ neighbors in $\F_{i_0+2}$ and $|M'_1|=z_1(i_0+1)$.
\item The vertices of $M'_2:=\{(A,B)\mid[i_0]\subseteq B,\ i_0+1\notin B,\ \min(A)=i_0+2\}$ have $2a$ neighbors in $\F_{i_0+2}$ and $|M'_2|=z_3(i_0+1)$.
\item Every other vertex outside $\F_{i_0+2}$ has more than $2a$ neighbors in $\F_{i_0+2}$.
\end{itemize}

Hence the number of vertices with $a+1$ neighbors in $\F_{i_0+1}$ is equal to the number of vertices with $a+1$ neighbors in $\F_{i_0+2}$. We have seen in the proof of (a) that this number is different from the number of vertices that have $a+1$ neighbors in $\F_{i_0}$. Thus $\F_{i_0}$ and $\F_{i_0+2}$ are not equivalent. In order to see that $\F_{i_0+1}$ and $\F_{i_0+2}$ are not equivalent, we consider two cases. We remark that this case is slightly more complicated, since for each $j$, the number of vertices with exactly $j$ neighbors in $\F_{i_0+1}$ is equal to the number of vertices with exactly $j$ neighbors in $\F_{i_0+2}$.

Case 1. $a\ge 2$, that is $a+1<2a$.

Then the set of vertices with $2a$ neighbors in $\F_{i_0+1}$ or $\F_{i_0+2}$ is $M_2$ respectively $M'_2$.
The elements of $\F_{i_0+1}$ that have at least one neighbor in $M_1$ are $(A',B')$ with $[i]\cup\{n\}\subseteq B'$ and $\min(A')=i$. Depending on $n\in A'$ or $n\notin A'$, the number of neighbors of $(A',B')$ in $M_1$ is different. The elements of $\bar\F_{i_0+2}$ that have a neighbor in $M'_2$ are $(A',B')$ with
$\min A'=i-1,\ [i-1]\subseteq B',\ i\notin B'$. By symmetry all these vertices have the same number of neighbors in $M_2$. This implies that $\F_{i_0+1}$ and $\F_{i_0+2}$ are not equivalent.

Case 2. $a=1$, so $a+1=2a$.

Then $M_1\cup M_2$ consists of the vertices with $2a=a+1$ neighbors in $\F_{i_0+1}$, and $M'_1\cup M'_2$ consists of the vertices with $2a=a+1$ neighbors in $\F_{i_0+2}$.

We have $|M_1|=6$ and $|M_2|=3$ where $M_1$ consists of $(A,B)$ with $A=\{c\}$, $B=[n-4]\cup\{c,n\}$, or $A=\{n \}$, $B=[n-4]\cup\{c,n\}$, and $M_2$ consists of the vertices $(A,B)$ with $A=\{c\}$ and $B=[n]\setminus\{n-4,n\}$, where $n-3\le c\le n-1$. We see that the first three vertices of $M_1$ have no neighbor in $M_1\cup M_2$.

The set $M'_1$ consists of the six vertices $B=[n-4]\cup G$ and $A=\{c\}\subseteq G$ where $G$ is $2$-subset of $\{n-2,n-1,n\}$, and $M_2$ consists of the three vertices with $B=[n-5]\cup G$ where $G=\{n-2,n-2,n-1\}$ and $A=\{c\}$ with $c\in G$. Hence every vertex of $M'_1\cup M'_2$ has at least one neighbor in $M'_1\cup M'_2$.

We have shown that graphs induced on $M_1\cup M_2$ and on $M'_1\cup M'_2$ are not isomorphic, and this implies that the sets $\F_{i_0+1}$ and $\F_{i0+2}$ are not equivalent.
\end{proof}

{\bf Notation}. For positive integers $n,a,b$ with $a+b<n$ and $a<\frac n2<b$ we denote by $f(n,a,b)$ the largest cardinality of the independent sets that have been constructed in Examples \ref{examp}, that is
\begin{align*}
f(n,a,b)&=\max\{|\F_j(n,a,b)|:0\le j\le 2b-n+1\}
\\
&=\binom{n-1-i}{b-i}\binom{b-i}{a}+\binom{n-b+a-1}{a-1}\left(\binom{n}{b-a}-\binom{n-i}{b-a-i}\right)
\end{align*}
where
\begin{align}\label{eq_def_i}
i:=\max\{0,\left\lceil b-1-\frac{(n-b)(n-b-1)}{a}\right\rceil\} .
\end{align}

\section{The cycle method}\label{cyclic}

In this section we consider the graphs $\Gamma(n,\{a,b\})$ with $a<\frac n2<b$ and $a+3b\le 2n$. The number $i_0$ defined in \eqref{eq_def_i0} is negative and hence $i$ defined in \eqref{eq_def_i} is zero, so the results in Section 3 show that $\F_0(n,a,b)$ is an independent set with
\begin{align*}
|\F_0(n,a,b)|=f(n,a,b)=\bino{n-1}{b}\bino{b}{a}.
\end{align*}
We now show that there is in fact no larger independent set.

\begin{theorem}\label{TheoremCycle}
For positive integers $a$, $b$ and $n$ with $n<2b$ and $a+3b\le 2n$ we have $$\alpha(\Gamma(n,\{a,b\})=\bino{n-1}{b}\bino{b}{a}.$$
\end{theorem}
\begin{proof}
The above example shows that the independence number can not be smaller. To show that equality holds, we consider $M:=\{0,1,\dots,n-1\}$ and an independent set $\F$ of flags of $M$ of type $\{a,b\}$. Put $s:=2b-n$ and notice that $a<2s+a\le b$. Let $X$ be the set consisting of all pairs $(f,c)$ with vertices $f=\{A,B\}\in\F$ and $n$-cycles $c=(m_0,\dots,m_{n-1})$ of the elements of $M$ such that there exists an index $i$ with $B=\{m_{i+j}\mid 1\le j\le b\}$ and $A=\{m_{i+s+j}\mid 1\le j\le a\}$, where indices are modulo $n$.

Each $f\in\F$ occurs in precisely $a!(b-a)!(n-b)!$ pairs of $X$. Therefore we have $|X|=|\F|a!(b-a)!(n-b)!$.

Now consider an $n$-cycle $c=(m_0,\dots,m_{n-1})$ and let $t$ be the number of pairs of $X$ in which $c$ occurs. Let these pairs be $(f_i,c)$, $1\le i\le t$, where $f_i=\{A_i,B_i\}$ with $|A_i|=a<b=|B_i|$. We want to show that $t\le n-b$. For this we may assume that $t\ge 1$. Consider $1\le i,j\le t$. As $f_i$ and $f_j$ are vertices that are not connected by an edge we have $|B_i\cap B_j|>2b-n$ or $A_i\cap B_j\not=\emptyset$ or $A_j\cap B_i\not=\emptyset$. Since $b\ge 2s+a$ and $(f_i,c),(f_j,c)\in X$, we see that each of $A_i\cap B_j\not=\emptyset$ and $A_j\cap B_i\not=\emptyset$ implies that $|B_i\cap B_j|\ge s+1>2b-n$. Hence $|B_i\cap B_j|>2b-n$ in any case. Thus any two of the sets $\bar B_i:=[n]\setminus B_i$ have non-empty intersection. As in the classical proof of Katona, see Lemma 2.14.1 in \cite{Godsil&Meagher}, it follows that $t\le n-b$. As there are $(n-1)!$ $n$-cycles, it follows that $|X|\le (n-1)!\cdot (n-b)$.

Comparing the lower and upper bound for $|X|$ implies the desired result.
\end{proof}

Using Proposition \ref{projectingspecial}, we find the following corollary, which is Theorem \ref{main1} from the introduction in a slightly different formulation.

\begin{corollary}
Let $a$, $b$ and $n$ be positive integers with $n<2b$ and $a+3b\le 2n$. Then for every subset $T$ of $[a-1]$ with we have
\begin{align*}
\alpha(\Gamma(n,T\cup\{a,b\})=\bino{n-1}{b}\bino{b}{a}\cdot |V\Gamma(a,T)|.
\end{align*}
\end{corollary}

\begin{remarks}
\begin{enumerate}
\item The set $\F_0(n,a,b)$ is in general not the only largest independent set of $\Gamma(n,\{a,b\})$. For example for $n=7$ and type $\{2,4\}$ there are 14 non equivalent independence sets of size $90$ as can be verified by computer and $\F_0(7,2,4)$, $\F_1(7,2,4)$ and $\bar\F_2(7,2,4)$ are among these.
\item For $(n,a,b)=(9,1,6)$ the condition $a+3b\le 2n$ is not fulfilled, but it can be verified by computer that $\alpha(\Gamma(n,\{a,b\})=\bino{n-1}{b}\bino{b}{a}$.
\end{enumerate}
\end{remarks}

\section{The chromatic number of $\Gamma(n,\{1,n-2\})$}\label{plus2}

In this section we classify the largest independent sets of $\Gamma(n,\{1,n-2\})$ for $n\ge 5$ using induction on $n$. The induction step can be proved  without significantly larger effort for more general graphs. We formulate this in Proposition \ref{proposition}, which is the main result of this section.

Recall that we have defined $f(n,a,b)$ for positive integers $n,a,b$ with $n>a+b$ and $a<\frac n2<b$ in the end of Section 3 and that $f(n,a,b)$ is the cardinality of the largest independent sets of $\Gamma(n,\{a,b\})$ that were constructed in Example \ref{examp}.

\begin{lemma}\label{fnabidenties}
For integers $n,a,b$ with $n\ge a+b+1$ and $a<\frac n2<b$ we have
\begin{align*}
f(n,a,b)&=f(n-1,a,b-1)+\binom{n-1}{b-1}\binom{b-1}{a-1}
\end{align*}
and with $i$ as in \eqref{eq_def_i} we have
\begin{align*}
n\cdot f(n-1,a,b-1)-(b-a)f(n,a,b)=\frac{[(n-b)^2+a(i-b)](n-1-i)!}{a!(n-b)!(b-i-a)!}
\end{align*}
\end{lemma}
\begin{proof}
This follows from straightforward calculations. For the second equation, one may apply \eqref{eq_good_formul_Fi} using that $f(n,a,b)=|\F_i(n,a,b)|$ and $f(n-1,a,b-1)=|\F_{i-1}(n-1,a,b-1)|$.
\end{proof}

Let $a$, $b$ and $n$ be positive integers with $n\ge a+b+1$ and $a<\frac n2<b$. Let $M$ be a set with $n$ elements. In this section an independent set $\F$ of $\Gamma(M,\{a,b\})$ will be called an independent set of \emph{standard type} if $|\F|=f(n,a,b)$ and if there exists a graph isomorphism $\mu$ of $\Gamma(M,\{a,b\})$ onto $\Gamma(n,\{a,b\})$ such that $\mu(\F)$ is equal to one of the independent sets of $\Gamma(n,\{a,b\})$ that were constructed in Example \ref{examp}.

\begin{proposition}\label{proposition}
Let $a,b,n\ge 1$ be integers with $n\ge a+b+1$ and $a<\frac n2<b$. Define $i$ as in \eqref{eq_def_i} and assume that
\begin{align}\label{condition_on_n}
n\ge\frac{[(n-b)^2+a(i-b)](n-1-i)!}{a!(n-b)!(b-i-a)!}+3a+1.
\end{align}
Assume furthermore that $f(n-1,a,b-1)$ is the independence number of $\Gamma(n-1,\{a,b-1\})$ and that every independent set of $\Gamma(n-1,\{a,b-1\})$ is of standard type. Then $f(n,a,b)$ is the independence number of $\Gamma(n,\{a,b\})$ and every independent set of $\Gamma(n,\{a,b\})$ is of standard type.
\end{proposition}

For the proof of the proposition we define $M:=[n]$ and consider an independent set $\F$ of $\Gamma(M,\{a,b\})$ with $|\F|=f(n,a,b)$. Recall that the vertices of $\Gamma(n,\{a,b\})$ are flags of type $\{a,b\}$, which we denote as ordered pairs $(A,B)$ where $A\subseteq B\subseteq M$, $a=|A|$ and $b=|B|$. In order to prove the proposition, we have to show that $\F$ is of standard type. This will be done in several lemmas in which we use for $c\in M$ the notation
\[
\F_c:=\{(A,B)\in \F\mid c\in B\setminus A\}
\]
and
\[
\hat\F_c:=\{(A,B\setminus\{c\})\mid (A,B)\in\F_c\}.
\]

\begin{lemma}\label{indhypo}
If the set $S$ has cardinality $n-1$, then the graph $\Gamma(S,\{a,b-1\})$ has independence number $f(n-1,a,b-1)$ and its largest independent set are of standard type.
\end{lemma}
\begin{proof}
This follows from that the hypotheses of Proposition \ref{proposition} since $\Gamma(S,\{a,b-1\})$ and $\Gamma(n-1,\{a,b-1\})$ are isomorphic graphs.
\end{proof}

\begin{lemma}\label{hatFcistbeispiel}
For every element $c\in M$, the set $\hat\F_c$ is an independent set of the graph $\Gamma(M\setminus\{c\},\{a,b-1\})$ with $|\F_c|\le f(n-1,a,b-1)$.
\end{lemma}
\begin{proof}
Since $\F$ is an independent of $\Gamma(M,\{a,b\})$, then $\F_c$ is an independent set of $\Gamma(M\setminus\{c\},\{a,b-1\})$. Apply Lemma \ref{indhypo}.
\end{proof}

\begin{lemma}\label{summederF_i}
We have
\begin{align}\label{wenigec}
\sum_{c\in M}\left(f(n-1,a,b-1)-|\F_c|\right)=\frac{[(n-b)^2+a(i-b)](n-1-i)!}{a!(n-b)!(b-i-a)!}.
\end{align}
\end{lemma}
\begin{proof}
Every element of $\F$ is a member of $\F_c$ for $b-a$ elements $c$ of $M$. Therefore $|\F|(b-a)=\sum_{c\in M}|\F_c|$. Since $|\F|=f(n,a,b)$, the claim follows from Lemma \ref{fnabidenties}.
\end{proof}

\begin{lemma}\label{standardtype}
If there exists an element $x\in M$ such that $x\in B$ for all elements $(A,B)$ of $\F$, then $\F$ is of standard type.
\end{lemma}
\begin{proof}
Lemma \ref{hatFcistbeispiel} gives $|\F_x|\le f(n-1,a,b-1)$. Since $x\in B$ for all $(A,B)\in\F$, then each element $(A,B)$ of $\F\setminus\F_x$ satisfies $x\in A$. This implies that $|\F\setminus\F_x|\le {n-1\choose b-1}\bino{b-1}{a-1}$. Hence
\begin{align*}
|\F|&=|\F_x|+|\F\setminus\F_x|\le f(n-1,a,b-1)+{n-1\choose b-1}\binom{b-1}{a-1}=f(n,a,b)
\end{align*}
where we have used Lemma \ref{fnabidenties} in the last step. Since $|\F|=f(n,a,b)$, it follows that $|\F_x|=f(n-1,a,b-1)$ and that $\F\setminus \F_x$ consists of all flags $(A,B)$ of $M$ of type $\{a,b\}$ with $x\in A$. Since $|\F_x|=|\hat\F_x|$, Lemma \ref{indhypo} shows that the independent set $\hat\F_x$ of $\Gamma(M\setminus\{x\},\{a,b-1\})$ is of standard type. From the construction in Example \ref{examp} it follows that $\F$ is of standard type.
\end{proof}

\begin{lemma}\label{onehat}
Suppose that $c$ is an element of $M$ with $|\F_c|=f(n-1,a,b-1)$.
\begin{enumerate}
\renewcommand{\labelenumi}{\rm(\alph{enumi})}
\item There exists a unique index $\hat c\not=c$ such that $\hat c\in B$ for all $(A,B)\in \F_c$.
\item Every vertex $(A,B)$ of $\Gamma(M,\{a,b\})$ with $\hat c\in A$ and $c\in B\setminus A$ lies in $\F_c$.
\item If $(A,B)\in \F$, then $\hat c\in B$ or $c\in A$.
\end{enumerate}
\end{lemma}
\begin{proof}
Since $|\hat\F_c|=|\F_c|=f(n-1,a,b-1)$, Lemma \ref{indhypo} shows that $\hat \F_c$ is an independent set of $\Gamma_c:=\Gamma(M\setminus\{c\},\{a,b-1\})$ of standard type. Inspection of the independent sets of standard type shows that there exists a unique element $\hat c\in M\setminus\{c\}$ such that $\hat c\in B'$ for every element $(A,B')$ of $\hat\F_c$. This proves (a). The inspection also shows that $(A,B')\in\F_c$ for every vertex $(A,B')$ of $\Gamma_c$ with $\hat c\in A$. This proves (b).

We prove (c) indirectly. Assume that $\F$ contains a vertex $(A,B)$ with $\hat c\notin B$ and $c\notin A$.
Then (a) shows that $(A,B)\notin \F_c$. Since $c\notin A$, this implies that $c\notin B$. Hence $B\subseteq M\setminus \{c,\hat c\}$. Choose a set $T$ with $A\subseteq T\subseteq B$ and $|T|=n-b$ and put $B_0:=M\setminus T$. Then $|B_0|=b$ and $|B_0\setminus B|=|B\setminus B_0|=|T|= n-b=a+1$. Also $c,\hat c\in B_0$. Since $n-b\ge a+1$, there exists a set $A_0\subseteq B_0$ of cardinality $a$ satisfying $\hat c\in A_0$ and $c\notin A_0$. From (b) we see that $(A_0,B_0)\in \F$. We have $B\cup B_0=M$ as well as $A\cap B_0=A_0\cap B=\emptyset$. Hence $(A_0,B_0)$ and $(A,B)$ are adjacent vertices. As both vertices lie in $\F$ and since $\F$ is an independent set, this is the desired contradiction.
\end{proof}

\textsc{Notation}. For every $c\in M$ with $|\F_c|=f(n-1,a,b-1)$ we denote by $\hat c$ the element provided by the previous lemma.

\begin{lemma}\label{twohat}
Let $C$ be the set of all elements $c\in M$ with $|\F_c|=f(n-1,a,b-1)$.
\begin{enumerate}
\renewcommand{\labelenumi}{\rm(\alph{enumi})}
\item If there exist $d\in M$ and $T\subseteq C$ with $|T|>a$ and $\hat c=d$ for all $c\in T$, then $d\in B$ for all $(A,B)\in \F$.
\item If $c_1,c_2\in C$, then $\hat c_1=\hat c_2$ or $\hat c_1=c_2$ or $\hat c_2=c_1$.
\end{enumerate}
\end{lemma}
\begin{proof}
(a) Assume on the contrary that $\F$ contains an element $(A,B)$ with $d\not\in B$. Then Lemma \ref{onehat} (c) shows $T\subseteq A$. As $|T|>a$, this is impossible.

(b) Assume on the contrary that $c_1,\hat c_1,c_2,\hat c_2$ are four distinct elements. Since $n-b\ge a+1\ge 2$ we see that $M$ has subsets $B_1$ and $B_2$ of cardinality $b$ with $c_1,\hat c_1\in B_1\setminus B_2$ and $c_1,\hat c_2\in B_2\setminus B_1$ and $B_1\cup B_2=M$. Then $|B_1\setminus B_2|=n-b\ge a+1$ and hence there exists a set $A_1\subseteq B_1\setminus B_2$ with $|A|=a$ and $\hat c_1\in A_1$ and $c_1\notin A_1$. The same argument shows that there exists $A_2\subseteq B_2$ with $|A_2|=a$ and $\hat c_2\in A_2$ and $c_2\notin A_2$. Lemma \ref{onehat} (b) shows that $(A_1,B_1),(A_2,B_2)\in\F$. But these two vertices are adjacent, contradiction.
\end{proof}

\begin{lemma}\label{atleastfour}
Suppose that there exists a set $T\subseteq M$ with $|T|=3a+1$ and $|\F_c|=f(n-1,a,b-1)$ for all $c\in T$. Then there exists an element $x\in M$  with $x\in B$ for all $(A,B)\in\F$.
\end{lemma}
\begin{proof}
This follows from Lemma \ref{twohat}, if there are $a+1$ elements $c\in T$ with the same $\hat c$. We may thus assume that this does not occur. As $|T|>3a$, this implies that there exits $c_1,c_2,c_3,c_4\in T$ such that $\hat c_1,\hat c_2,\hat c_3,\hat c_4$ are pairwise distinct. Lemma \ref{twohat} (b) shows that we may assume that $\hat c_2=c_1$. Since $\hat c_3\not=\hat c_4$, we may also assume that $\hat c_3\not=c_2$. Then $c_2,c_3,\hat c_2,\hat c_3$ are mutually distinct. Lemma \ref{twohat} (b) gives a contradiction.
\end{proof}

{\bf Proof of Proposition \ref{proposition}}. Using \eqref{condition_on_n} we see from Lemma \ref{hatFcistbeispiel} and \ref{summederF_i} that $|\F_c|=f(n-1,a,b-1)$ for at least $3a+1$ elements $c$ of $M$. Then Lemma \ref{atleastfour} shows that there exists an element $x\in M$ with $x\in B$ for all $(A,B)\in\F$. Lemma \ref{standardtype} now proves the proposition.

REMARK. The proposition may be used as an induction step $n-1\to n$ to determine the independence number of $\Gamma(n,a,n-s)$ for given $a$ and $s$ (with $\frac n2>s\ge a+1\ge 2$) when $n$ satisfies \eqref{condition_on_n} for $a$ and $b=n-s$, that is for
\begin{align}
n\ge\frac{[s^2-a(1+x)](s+x)!}{a!s!(1+x-a)!}+3a+1.
\end{align}
where $x=\lfloor(s(s-1)/a\rfloor$. For $a=1$ and $s=2$, this is $n\ge 10$, and so the induction basis requires to show that the largest independent sets of $\Gamma(9,\{1,7\})$ are of standard type. This can be done by computer. For all other values of $a$ and $s$, the integer $n$ must already be so large that I did not succeed to verify that induction hypothesis by computer. For example for $a=2$ and $s=3$, we already need $n\ge 37$, so the induction hypothesis requires to show that the independence number of $\Gamma(36,\{1,33\})$, which is a graph with 235620 vertices, is 58947. So we obtain a theorem only in the case when $a=1$ and $s=2$, that is $b=n-2$.

\begin{theorem}\label{nhochplus2}
For all integers $n\ge 5$ the graph $\Gamma(n,\{1,n-2\})$ has independence number $\bino{n}{3}+2$ and every independent set of $\Gamma(n,\{1,n-2\})$ of cardinality $\bino{n}{3}+2$ is equivalent to $\F_{n-5}(n,1,n-2)$ or $\F_{n-4}(n,1,n-2)$ or $\bar\F_{n-3}(n,1,n-2)$ as defined in Example \ref{examp}.
\end{theorem}
\begin{proof}
Using induction on $n$ and Proposition \ref{proposition}, we only have to verify that the statement is true for $n\in\{5,6,7,8,9\}$. This can easily be done by computer for example using the package grape in GAP.
\end{proof}

\section{Open problems}

A natural question is probably if not to determine the independence number of all graphs $\Gamma(n,M)$ to do this for the case $|M|=2$. The only open case here for $n\le 8$ is $n=8$ and $M=\{2,5\}$. We give some information on this case and some more special problems below.

\begin{lemma}
If $\alpha$ is the chromatic number of $\Gamma(8,\{2,5\})$, then $230\le \alpha\le 240$.
\end{lemma}
\begin{proof}
The set $\F_1(8,2,5)$ is an independent set of $\Gamma(8,\{2,5\})$ with $230$ elements. This gives the lower bound. For the upper bound we use that $\Gamma(7,\{2,4\})$ has independence number 90 by Theorem \ref{TheoremCycle}. Suppose $\F$ is an independent set of $\Gamma(8,\{2,5\})$. For each $c\in[8]$, let $\F_c$ be the set of elements $(A,B)$ of $\F$ with $c\in B\setminus A$. As in the proof of Lemma \ref{summederF_i} we see that $\sum_{i=1}^8|\F_c|=(5-2)\cdot|\F|$. Each set $\F_c$ is an independent set of the graph $\Gamma([8]\setminus\{c\},\{2,5\})$ and hence has at most $90$ elements. It follows that $3|\F|\le 8\cdot 90$, so $|\F|\le 240$.
\end{proof}

Remark. The smallest eigenvalue of the adjacency matrix $A$ of $\Gamma(8,\{2,5\})$ is $-\frac32(1+\sqrt{17})$, so the Hoffman bound gives $\alpha\le 257$ for the independence number. If one considers the coherent configuration corresponding to $\Gamma(8,\{2,5\})$, then one can apply semidefinite programming, but in this case it gives nothing better than the Hoffman bound. It can be checked that $A$ has the same number of positive and negative eigenvalues, so the inertia bound applied to $A$ does not give a better bound, but I did not check other matrices. Finally, since the independence number is larger than $560/3$, also the clique-coclique bound does not help to find a better bound for $\alpha$

Problem 1: Determine the chromatic number of $\Gamma(8,T)$ in the cases that $T$ is one of the sets $\{2,5\}$, $\{1,3,5\}$ and $\{1,4,5\}$. These are all open problems for $n=8$ up to duality.

Problem 2: Determine the chromatic number of $\Gamma(9,T)$ for $T=\{1,6\}$ and for $T=\{2,6\}$. These are the only open cases for $n=9$ and $|T|=2$.

Problem 3: It can be checked by computer that every independent set of $\Gamma(7,\{1,4\})$ is equivalent to $\F_0(7,\{1,4\})$. Find an computer-free argument for this fact.

Problem 4: Are there constructions of independent sets of $\Gamma(n,\{a,b\})$ with $a+b<n$ and $a<\frac n2<b$ that are larger than the ones constructed in Examples \ref{examp}.

\end{document}